\documentclass[a4paper,xcolor=table,10pt]{article}

\usepackage{graphics,psfrag}
\usepackage{latexsym}
\usepackage{graphicx}
\usepackage{relsize}
\catcode`@=11
\renewcommand{\@begintheorem}[2]{\it \trivlist
      \item[\hskip \labelsep{\bf #1\ #2{\rm :}}]}
\renewcommand{\@opargbegintheorem}[3]{\it \trivlist
      \item[\hskip \labelsep{\bf #1\ #2\ {\rm (#3)\/:}}]}
\def\@sect#1#2#3#4#5#6[#7]#8{\ifnum #2>\c@secnumdepth
     \def\@svsec{}\else
     \refstepcounter{#1}\edef\@svsec{\csname the#1\endcsname{.}\hskip 1em }\fi
     \@tempskipa #5\relax
      \ifdim \@tempskipa>\z@
        \begingroup #6\relax
          \@hangfrom{\hskip #3\relax\@svsec}{\interlinepenalty \@M #8\par}
        \endgroup
       \csname #1mark\endcsname{#7}\addcontentsline
         {toc}{#1}{\ifnum #2>\c@secnumdepth \else
                      \protect\numberline{\csname the#1\endcsname}\fi
                    #7}\else
        \def\@svsechd{#6\hskip #3\@svsec #8\csname #1mark\endcsname
                      {#7}\addcontentsline
                           {toc}{#1}{\ifnum #2>\c@secnumdepth \else
                             \protect\numberline{\csname the#1\endcsname}\fi
                       #7}}\fi
     \@xsect{#5}}

\@addtoreset{equation}{section}
\newcommand{\Delete}[1]{}

\usepackage{amsmath,amssymb,amsthm}

\theoremstyle{plain}
\newtheorem{Thm}{Theorem}[section]
\newtheorem{Lem}[Thm]{Lemma}
\newtheorem{Prop}[Thm]{Proposition}

\newtheorem*{Defi}{Definition}
\newtheorem{Prob}[Thm]{Problem}

\usepackage{graphpap}
\usepackage{pstricks}
\usepackage{pst-node}

\newcommand{\bR}{\ensuremath{\mathbb{R}}}

\newcommand{\cF}{\ensuremath{\mathcal{F}}}

\newcommand{\cH}{\ensuremath{\mathcal{H}}}

\begin{document}

\title{The competition-common enemy graphs 
of digraphs satisfying Conditions $C(p)$ and $C'(p)$}

\author{\begin{tabular}{c}
{\sc Yoshio SANO}
\thanks{This work was supported by Priority Research Centers Program 
through the National Research Foundation of Korea (NRF) 
funded by the Ministry of Education, Science and Technology 
(2009-0094069).}\\
\\
Pohang Mathematics Institute \\
POSTECH, Pohang 790-784, Korea\\
{\tt ysano@postech.ac.kr}
\end{tabular} }

\date{}

\maketitle

\begin{abstract}
S. -R. Kim and F. S. Roberts (2002) introduced 
the following conditions $C(p)$ and $C'(p)$ for digraphs 
as generalizations of the condition for digraphs to be semiorders. 
The condition $C(p)$ (resp. $C'(p)$) is: 
For any set $S$ of $p$ vertices in $D$, 
there exists $x \in S$ such that 
$N^+_D(x) \subseteq N^+_D(y)$ 
(resp. $N^-_D(x) \subseteq N^-_D(y)$) 
for all $y \in S$, 
where $N^+_D(x)$ (resp. $N^-_D(x)$) 
is the set of out-neighbors (resp. in-neighbors) of $x$ in $D$. 
The {\it competition graph} 
of a digraph $D$ is the (simple undirected) graph 
which has the same vertex set as $D$ 
and has an edge between two distinct vertices $x$ and $y$ 
if $N^+_D(x) \cap N^+_D(y) \neq \emptyset$. 
Kim and Roberts characterized the competition graphs of digraphs which 
satisfy Condition $C(p)$. 

The {\it competition-common enemy graph} 
of a digraph $D$ is the graph 
which has the same vertex set as $D$ 
and has an edge between two distinct vertices $x$ and $y$ 
if it holds that both 
$N^+_D(x) \cap N^+_D(y) \neq \emptyset$ 
and 
$N^-_D(x) \cap N^-_D(y) \neq \emptyset$. 
In this note, we characterize the competition-common enemy graphs 
of digraphs satisfying Conditions $C(p)$ and $C'(p)$. 
\end{abstract}


\newpage
\section{Introduction}

J. E. Cohen \cite{Cohen} introduced the notion of 
a competition graph in 1968 in connection with a problem in ecology. 
The {\it competition graph} $C(D)$ of a digraph $D$ 
is the (simple undirected) graph $G=(V,E)$ which 
has the same vertex set as $D$ 
and has an edge between two distinct vertices 
$x$ and $y$ if and only if $N^+_D(x) \cap N^+_D(y) \neq \emptyset$, 
where $N^+_D(x) := \{v \in V(D) \mid (x,v) \in A(D) \}$ 
is the set of out-neighbors of $x$ in $D$. 
It has been one of important research problems in the study of competition 
graphs to characterize the competition graphs of digraphs satisfying 
some conditions. 

\begin{Defi}
{\rm 
A digraph $D=(V,A)$ is called a 
{\it semiorder} if there exist a real-valued function 
$f:V \to \bR$ on the set $V$ and a positive real number $\delta \in \bR$ 
such that $(x,y) \in A$ if and only if $f(x) > f(y) + \delta$. 

A digraph $D=(V,A)$ is called an 
{\it interval order} if there exists 
an assignment $J:V \to 2^{\bR}$ 
of a closed real interval $J(x) \subset \bR$
to each vertex $x \in V$ 
such that $(x,y) \in A$ if and only if $\min J(x) > \max J(y)$. 
}
\qed
\end{Defi}

Kim and Roberts characterized the competition graphs of semiorders 
and interval orders as follows: 

\begin{Thm}[\cite{KimRoberts}]\label{thm:KR}
Let $G$ be a graph. Then the following are equivalent. 
\begin{itemize}
\item[{\rm (a)}]
$G$ is the competition graph of a semiorder, 
\item[{\rm (b)}]
$G$ is the competition graph of an interval order, 
\item[{\rm (c)}]
$G = K_r \cup I_q$ where if $r \geq 2$ then $q \geq 1$. 
\qed
\end{itemize}
\end{Thm}

\noindent
Moreover, 
Kim and Roberts \cite{KimRoberts} introduced 
some conditions, which are called Condition $C(p)$ and Condition $C'(p)$, 
for digraphs 
as generalizations of the condition for digraphs to be semiorders,  
and they gave a characterization of the competition graphs of digraphs 
satisfying Condition $C(p)$ to show Theorem \ref{thm:KR} as its corollary. 

D. D. Scott \cite{Scott} 
introduced the {\it competition-common enemy graph} of a digraph in 1987 
as a variant of competition graph. 
The {\it competition-common enemy graph} 
of a digraph $D$ is the graph 
which has the same vertex set as $D$ 
and has an edge between two distinct vertices $x$ and $y$ 
if it holds that both 
$N^+_D(x) \cap N^+_D(y) \neq \emptyset$ 
and 
$N^-_D(x) \cap N^-_D(y) \neq \emptyset$, where 
$N^-_D(x) := \{v \in V(D) \mid (v,x) \in A(D) \}$ 
is the set of in-neighbors of $x$ in $D$.

\begin{figure}[h]
\begin{center}
 \resizebox{10cm}{!}{
  \includegraphics{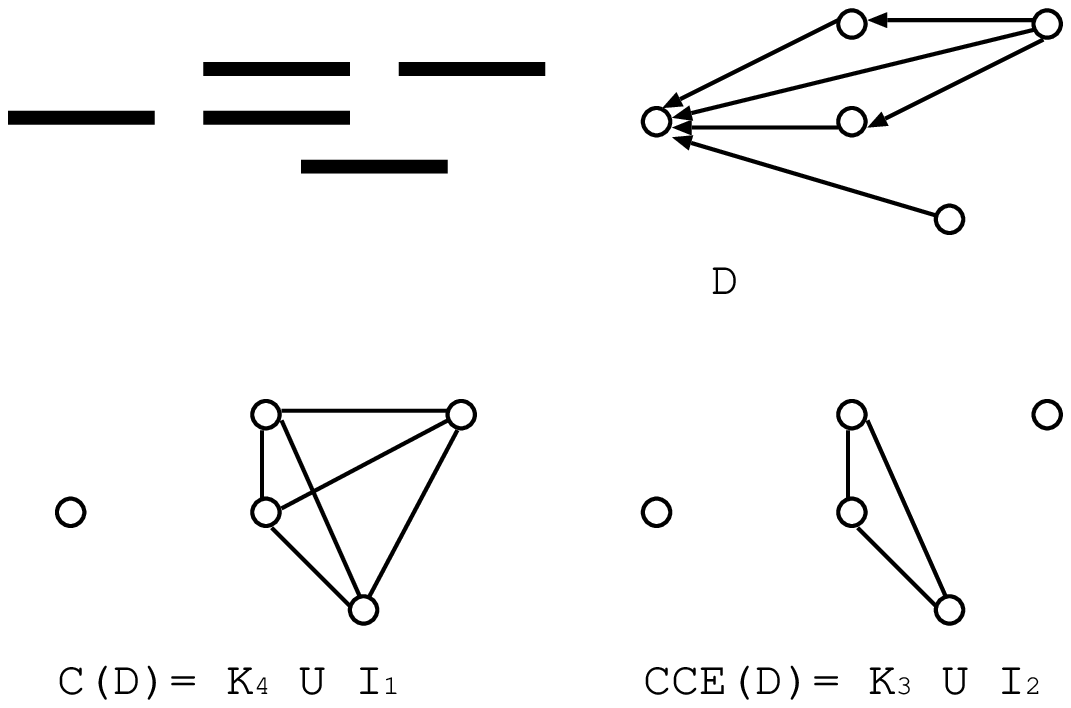}
 }
\end{center}
\caption{An interval order $D$, the competition graph $C(D)$, and 
the competition-common enemy graph $CCE(D)$}
\label{fig:interval}
\end{figure}

In this note, we characterize 
the competition-common enemy graphs of semiorders 
and interval orders as follows: 

\begin{Thm}\label{thm:main0}
Let $G$ be a graph. Then the following are equivalent. 
\begin{itemize}
\item[{\rm (a)}]
$G$ is the competition-common enemy graph 
of a semiorder, 
\item[{\rm (b)}]
$G$ is the competition-common enemy graph 
of an interval order, 
\item[{\rm (c)}]
$G = K_r \cup I_q$ where if $r \geq 2$ then $q \geq 2$. 
\qed
\end{itemize}
\end{Thm}

\noindent
Furthermore, 
we also characterize the competition-common enemy graphs 
of digraphs satisfying Conditions $C(p)$ and $C'(p)$.

\section{Main Results}
\subsection{Conditions $C(p)$ and $C'(p)$}

\begin{Defi}
{\rm 
Let $D$ be a digraph. 
For a set $S$ of vertices in $D$, we define the following: 
\begin{eqnarray*}
\cF^+_D(S) &:=& 
\{x \in S \mid N^+_D(x) \subseteq N^+_D(y) \text{ for all } y \in S \}, \\
\cF^-_D(S) &:=& 
\{x \in S \mid N^-_D(x) \subseteq N^-_D(y) \text{ for all } y \in S \}, \\
\cH^+_D(S) &:=& 
\{x \in S \mid N^+_D(x) \supseteq N^+_D(y) \text{ for all } y \in S \}, \\
\cH^-_D(S) &:=& 
\{x \in S \mid N^-_D(x) \supseteq N^-_D(y) \text{ for all } y \in S \}. 
\end{eqnarray*}
(Note that, in \cite{KimRoberts}, 
an element in $\cF^+_D(S)$ is called a {\it foot} of $S$ and 
an element in $\cH^+_D(S)$ is called a {\it head} of $S$.) 

\begin{figure}[h]
\begin{center}
 \resizebox{10cm}{!}{
  \includegraphics{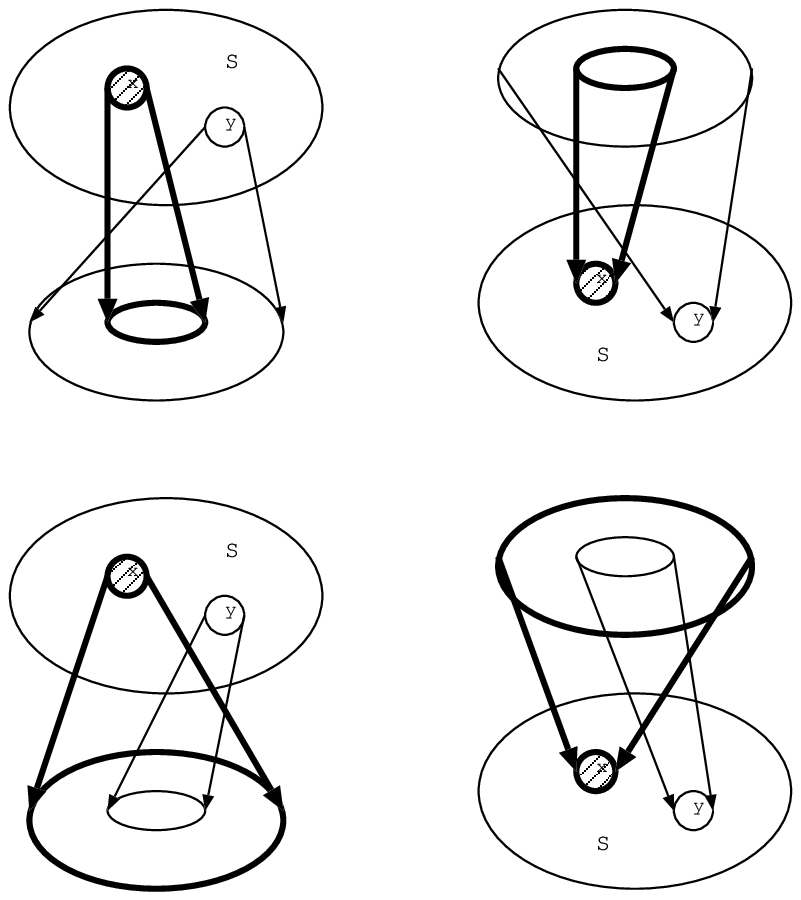}
 }
\end{center}
\caption{Elements $x$ in $\cF^+_D(S)$, $\cF^-_D(S)$, $\cH^+_D(S)$, 
and $\cH^-_D(S)$}
\label{fig:FS-HS}
\end{figure}

Let $p$ be a positive integer with $p \geq 2$. 
We say that {\it $D$ satisfies Condition $C(p)$} 
(resp. {\it Condition $C'(p)$}, {\it Condition $C^*(p)$}, 
{\it Condition ${C^*}'(p)$}) 
if the set $\cF^+_D(S)$ (resp. $\cF^-_D(S)$ $\cH^+_D(S)$ $\cH^-_D(S)$) 
is not empty 
for any set $S$ of $p$ vertices in the digraph $D$. 
}
\end{Defi}

\begin{Prop}[\cite{KimRoberts}]\label{prop:C(p)}
Let $2 \leq p < q$. 
If a digraph $D$ satisfies Condition $C(p)$, 
then the digraph $D$ also satisfies Condition $C(q)$. 
\end{Prop}

\begin{Lem}\label{lem:FD}
Let $D$ be a digraph and $T, U$ be sets of vertices in $D$. 
If $\cF^-_D(T) \cap U \neq \emptyset$, 
then $\cF^-_D(U) \subseteq \cF^-_D(T \cup U)$. 
\end{Lem}

\begin{proof}
Take $t \in \cF^-_D(T) \cap U$. 
Then $N^-_D(t) \subseteq N^-_D(t')$ for any $t' \in T \setminus U$. 
If $\cF^-_D(U)$ is empty, then the lemma trivially holds. 
So we assume that $\cF^-_D(U) \neq \emptyset$. 
Take any $u \in \cF^-_D(U)$. 
Then $N^-_D(u) \subseteq N^-_D(u')$ for any $u' \in U$. 
Since $t \in U$, we have $N^-_D(u) \subseteq N^-_D(t)$. 
Therefore, $N^-_D(u) \subseteq 
N^-_D(t')$ for any $t' \in T \setminus U$. 
Thus $N^-_D(u) \subseteq 
N^-_D(s)$ for any $s \in (T \setminus U) \cup U= T \cup U$. 
Hence the lemma holds. 
\end{proof}

\begin{Prop}\label{prop:C'(p)}
Let $2 \leq p < q$. 
If a digraph $D$ satisfies Condition $C'(p)$, 
then the digraph $D$ also satisfies Condition $C'(q)$. 
\end{Prop}

\begin{proof}
It is enough to show that 
$D$ satisfies Condition $C'(p+1)$. 
Let $S$ be any set of $p+1$ vertices of $D$, and 
let $T$ be a subset of $S$ with $|T|=p$. 
Then $\cF^-_D(T) \neq \emptyset$ 
since $D$ satisfies Condition $C'(p)$. 
Take an element $x$ in $\cF^-_D(T)$. 
Let $U$ be a subset of $S$ such that $|U|=p$ and $x \in U$. 
Since $p \geq 2$, it holds that $T \cup U = S$. 
By Lemma \ref{lem:FD}, 
we have $\cF^-_D(U) \subseteq \cF^-_D(T \cup U) = \cF^-_D(S)$. 
Since $D$ satisfies Condition $C'(p)$, 
$\cF^-_D(U) \neq \emptyset$. 
Thus $\cF^-_D(S)$ is not empty. 
\end{proof}

For a graph $G$, we denote the set of all isolated vertices in $G$ 
by $I_G$. Then the graph $G - I_G$ is the union of the nontrivial 
connected components of $G$.

\begin{Prop}\label{prop:clique}
Let $G$ be the competition-common enemy graph of 
a digraph $D$ which satisfies 
Conditions $C(p)$ and $C'(p)$ for some $p \geq 2$. 
Suppose that $G-I_G$ has at least $p$ vertices. 
Then $G-I_G$ is a clique of $G$. 
\end{Prop}

\begin{proof} 
Take any two vertices $a$ and $b$ in $G-I_G$. 
Then $a$ and $b$ are not isolated. 
Let $S$ be a set of $p$ vertices in $G-I_G$ 
containing the vertices $a$ and $b$. 
Since $D$ satisfies Conditions $C(p)$ and $C'(p)$, 
there exist $x \in \cF^+_D(S)$ and $y \in \cF^-_D(S)$. 
Note that $x$ and $y$ are not isolated vertices. 
Take $u \in N^+_D(x)$ and $v \in N^-_D(y)$. 
By Condition $C(p)$, we have $u \in N^+_D(a) \cap N^+_D(b)$. 
By Condition $C'(p)$, we have $v \in N^-_D(a) \cap N^-_D(b)$. 
Therefore $a$ and $b$ are adjacent in $G-I_G$. 
Hence the proposition holds. 
\end{proof}

\subsection{Classifiation}

\begin{Thm}\label{thm:loopless}
Let $G$ be a graph and $p \geq 2$. 
Suppose that $G-I_G$ has at least $p$ vertices. 
Then $G$ is the competition-common enemy graph of 
a loopless digraph satisfying Conditions $C(p)$ and $C'(p)$ 
if and only if 
$G = K_r \cup I_q$ with $r \geq p$ and $q \geq 2$. 
\end{Thm}

\begin{proof} 
Fisrt, we show the ``only if" part. 
Let $G$ be the competition-common enemy graph of 
a loopless digraph $D$ satisfying Conditions $C(p)$ and $C'(p)$. 
Proposition \ref{prop:clique} shows 
that $G = K_r \cup I_q$ with $r \geq p$ and $q \geq 0$. 
Suppose that $q = 0$ or $q=1$. 
Since $r \geq p$, 
by Propositions \ref{prop:C(p)} and \ref{prop:C'(p)}, 
$D$ satisfies Conditions $C(r)$ and $C'(r)$. 
Let $x \in \cF^+_D(S)$ and $y \in \cF^-_D(S)$ where $S:=V(G-I_G)$. 
Since $x$ and $y$ are not isolated in $G$, 
we have $N^+_D(x) \neq \emptyset$ and $N^-_D(y) \neq \emptyset$. 
Let $u \in N^+_D(x)$ and $v \in N^-_D(y)$. 
If $u=v$, then $(s,u) \in A(D)$ and $(u,s) \in A(D)$ for any $s \in S$. 
Let $S'$ be a set of $p$ vertices containing the vertex $u$. 
Note that $S' \setminus\{u\} \subseteq S$ since $q \leq 1$. 
By Condition $C(p)$, $\cF^+_D(S') \neq \emptyset$. 
If $u \in \cF^+_D(S')$, then 
we have $s \in N^+_D(u) \subseteq N^+_D(s)$ for $s \in S' \setminus\{u\}$, 
i.e., $(s,s) \in A(D)$, 
which contradicts that $D$ is loopless. 
If $s \in \cF^+_D(S')$ for some $s \in S' \setminus\{u\}$, then 
we have $u \in N^+_D(s) \subseteq N^+_D(u)$, i.e., $(u,u) \in A(D)$, 
which contradicts that $D$ is loopless. 
Therefore $u$ and $v$ must be distinct. 
Since $q \leq 1$, at least one of $u$ and $v$ is in $S=V(G-I_G)$. 
If $u \in S$, then 
we have $u \in N^+_D(x) \subseteq N^+_D(u)$, i.e., $(u,u) \in A(D)$. 
If $v \in S$, then 
we have $v \in N^-_D(y) \subseteq N^-_D(v)$, i.e., $(v,v) \in A(D)$. 
In any case, we reach a contradiction. 
Thus we have $q \geq 2$. 

Second, we show the ``if" part. 
Let $G = K_r \cup I_q$ with $r \geq p$ and $q \geq 2$.
Let $a$ and $b$ be distinct vertices in $I_q$. 
We define a digraph $D$ by 
$V(D):=V(G)$ and 
$A(D):= 
\{(a, x) \mid x \in V(K_r) \} 
\cup
\{(x, b) \mid x \in V(K_r) \} 
\cup \{(a,b)\}$. 
Then $D$ is loopless, $D$ satisfies Conditions 
$C(p)$ and $C'(p)$, and 
the competition-common enemy graph of $D$ is equal to $G$. 
\end{proof}

The {\it double competition number} $dk(G)$ of a graph $G$ 
is the minimum number $k$ such that $G$ with $k$ new isolated vertices 
is the competition-common enemy graph of an acyclic digraph. 


\begin{Thm}\label{thm:acyclic}
Let $G$ be a graph and $p \geq 2$. 
If $G$ is the competition-common enemy graph of an 
acyclic digraph $D$ satisfying Conditions $C(p)$ 
and $C'(p)$, then  
$G$ is one of the following graphs:
\begin{itemize}
\item[{\rm (a)}] 
$I_q$ $(q \geq 1)$, 
\item[{\rm (b)}] 
$K_r \cup I_q$ $(r \geq p$, $q \geq 2)$, 
\item[{\rm (c)}]
$H \cup I_q$ where $|V(H)| < p$, $I_H=\emptyset$, and 
$q \geq dk(H)$. 
\end{itemize}
\end{Thm}

\begin{proof}
Let $G$ be the competition-common enemy graph of 
an acyclic digraph $D$ satisfying 
Conditions $C(p)$ and $C'(p)$. 
If there is no nontrivial connected component in $G$, then (a) holds. 
Let $H$ be the union of all nontrivial connected components of $G$. 
Then we have $G = H \cup I_q$ with $q \geq 0$ and $I_H=\emptyset$. 
If $H$ has at least $p$ vertices, 
then it follows from Theorem \ref{thm:loopless} that 
$G=K_r \cup I_q$ with $r \geq p$ and $q \geq 2$, i.e., (b) holds. 
Suppose that 
the number of the vertices of $H$ is less than $p$. 
Since $G$ is the competition-common enemy graph 
of an acyclic digraph $D$, 
the double competition number 
$dk(G)$ of $G$ is equal to $0$. 
Therefore, there must be at least $dk(H)$ vertices in $I_q$. 
Hence (c) holds. 
\end{proof}

\subsection{Proof of Theorem \ref{thm:main0}}

\begin{proof}[Proof of Theorem \ref{thm:main0}]

(a) $\Rightarrow$ (b): 
Since semiorders are a special case of interval orders 
where every interval has the same length, 
(a) implies (b). 

(b) $\Rightarrow$ (c): 
We can easily check that any 
interval order satisfies Conditions 
$C(2)$ and $C'(2)$. 
By Theorem \ref{thm:acyclic} with $p=2$, 
we can conclude that if (b) then (c). 

(c) $\Rightarrow$ (a): 
Suppose that 
$G = I_q$ $(q \geq 1)$ or 
$G = K_r \cup I_q$ $(r \geq 2$, $q \geq 2)$. 
When $G = I_q$, 
we let $f_1(x):=0$ for any $x \in V(G)$ and let $\delta_1:=1$. 
Then $G$ is the competition-common enemy graph of the semiorder 
defined by $f_1$ and $\delta_1$. 
When $G = K_r \cup I_q$, 
we take a vertex $a$ in $I_q$, 
let $f_2(x):=0$ for any $x \in V(K_r)$, 
$f_2(a):=2$, and 
$f_2(b):=-2$ for any $b \in V(I_q) \setminus \{a\}$,  
and let $\delta_2:=1$. 
Then $G$ is the competition-common enemy graph of the semiorder 
defined by $f_2$ and $\delta_2$. 

Hence Theorem \ref{thm:main0} holds. 
\end{proof}

\section{Concluding Remarks}

In this section, 
we present some problems for further study. 

In Theorem \ref{thm:loopless}, 
we gave a characterization of the competition common-enemy 
graphs $G$ of digraphs satisfying Conditions $C(p)$ and $C'(p)$ 
if the number of the vertices of $G-I_G$ is at least $p$. 

\begin{Prob}
Characterize 
the competition-common enemy graphs $G$ 
of digraphs satisfying Conditions 
$C(p)$ and $C'(p)$ when 
the number of the vertices of $G-I_G$ is less than $p$. 
\end{Prob}

In this note, we didn't consider Conditions $C^*(p)$ and ${C^*}'(p)$. 

\begin{Prob}
Characterize 
the competition-common enemy graphs of digraphs satisfying Conditions 
$C^*(p)$ and ${C^*}'(p)$. 
\end{Prob}

Niche graphs are another variant of competition graphs 
and were introduced by 
C. Cable, K. F. Jones, J.R. Lundgren, and S. Seager \cite{CJLS}. 
The {\it niche graph} of a digraph $D$ 
is the graph which has the same vertex set as $D$ 
and has an edge between two distinct vertices $x$ and $y$ 
if 
$N^+_D(x) \cap N^+_D(y) \neq \emptyset$ 
or  
$N^-_D(x) \cap N^-_D(y) \neq \emptyset$. 

\begin{Prob}
What are the niche graphs of digraphs satisfying Conditions 
$C(p)$, $C'(p)$, $C^*(p)$, or ${C^*}'(p)$? 
\end{Prob}



\begin{thebibliography}{99}%


\bibitem {CJLS}
C. Cable, K. F. Jones, J.R. Lundgren, and S. Seager: 
Niche graphs, 
{\it Discrete Applied Mathematics} {\bf 23} (1989) 231--241.


\bibitem {Cohen}
J. E. Cohen: 
Interval graphs and food webs. A finding and a problem, 
RAND Corporation {\it Document 17696-PR}, Santa Monica, California (1968).

\bibitem {KimRoberts}
S. -R. Kim and F. S. Roberts:
Competition graphs of semiorders and Conditions $C(p)$ and $C^{*}(p)$,
{\it Ars Combinatoria} {\bf 63} (2002) 161--173.


\bibitem{LeeKim}
J. Y. Lee and S. -R. Kim: 
Competition graphs of acyclic digraphs satisfying condition {$C^*(p)$},
{\em Ars Combinatoria} {\bf 93} (2009) 321--332.

\bibitem{Scott}
D. Scott:
The competition-common enemy graph of a digraph,
{\it Discrete Applied Mathematics} {\bf 17} (1987) 269--280.


\end{thebibliography}
\end{document}